\documentclass{article}
\usepackage{amsmath}
\usepackage{amssymb}
\usepackage{amsfonts}
\usepackage[top=2cm,bottom=2cm,left=2.5cm,right=2cm]{geometry}

\usepackage{amsmath}
\usepackage{graphicx}
\usepackage[colorinlistoftodos]{todonotes}
\usepackage{amsthm}
\usepackage{bbm}
\usepackage{doc}

\newtheorem{theorem}{Theorem}[section]
\newtheorem{claim}[theorem]{Claim}

\newtheorem{lemma}[theorem]{Lemma}

\usepackage{url}
\usepackage[colorlinks=true, allcolors=blue]{hyperref}

\title{On the Joint Distribution of the Roots of Pairs of Polynomial Congruences}
\author{Sa'ar Zehavi\\
School of Mathematical Sciences, Tel Aviv University, Tel Aviv 69978, Israel\\saarzehavi@mail.tau.ac.il
}

\begin{document}
\maketitle
\begin{abstract}
    Let $f(x)\in\mathbb{Z}[x]$ be a primitive irreducible polynomial of degree greater than one. In~\cite{HOOLEY}, Hooley showed that the sequence $\frac{\mu}{n}$, where $f(\mu) = 0(n)$, ordered in the obvious way, is uniformly distributed modulo one.
    
    It is the goal of this paper to show that if $f(x),g(x)\in\mathbb{Z}[x]$ are a pair of primitive irreducible polynomials of degree greater than one, not necessarily distinct, then the sequence $(\frac{\mu}{n},\frac{\nu}{n})$, with $f(\mu) = 0(n)$ and $g(\nu) = 0(n)$, ordered in the obvious way, is uniformly distributed modulo one in the unit torus.
\end{abstract}

\section{Introduction}
A famous theorem of Hooley~\cite{HOOLEY} states that the roots of the reductions modulo $n$ of a primitive, irreducible polynomial of degree greater than one are uniformly distributed as $n$ varies through all integers. It is the goal of this paper to study an analogue of this question when one chooses a pair of primitive irreducible polynomials of degree greater than one instead of a single polynomial. What is their joint distribution like? Under what conditions are those distributions independent? It is the main theorem of this paper that under no constraints on the choice of the pair of polynomials, as long as they are primitive and irreducible, their roots modulo $n$ are jointly equidistributed as $n$ varies through all integers.

Let $f(x)\in\mathbb{Z}[x]$ be a primitive, irreducible polynomial. We consider its reductions modulo $n$, denoted $f_n(x)$. Denote by $r(n)$ the number of roots of $f_n(x)$:
\[
r(n) := \#\{m\in\mathbb{Z}/n\mathbb{Z}\mid f_n(m) = 0(n)\}.
\]
Let $(a_i^n)_{i=1}^{r(n)}$ be a sequence of representatives of the solutions of $f_n(x) = 0(n)$, taken from the interval $0,1,...,n-1$. We define the normalized sequence of representatives by
\[
(A_i^n)_{i=1}^{r(n)},\quad A_i^n := \dfrac{a_i^n}{n}.
\]
Clearly, $A_i^n\in [0,1)$ for all $i$ and all $n$. We define a series $X$ to be the concatenation of the sequences $A_i^n$, according to the left dictionary order $(n,i)$.
\[
X = (..., A_1^n, A_2^n, ..., A_{r(n)}^n, A_1^{n+1}, ...).
\]
Denote by $X_i$ the $i$'th element of the series $X$. In~\cite{HOOLEY}, Hooley shows that the series $X$ is uniformly distributed in $[0,1]$:
\begin{theorem}
\label{thm:hooley_main}
\cite[Theorem 2]{HOOLEY} Let $f(x)\in\mathbb{Z}[x]$ be a primitive, irreducible polynomial of degree greater than 1. Let $X$ be the series of normalized roots of the reductions modulo $n$ of $f(x)$ (as constructed above). Let $0\le \alpha < \beta\le 1$ be a pair of real numbers. Then
\[
\lim_{N\rightarrow\infty}\dfrac{\#\{i\in \mathbb{N}\mid 1\le i\le N, X_i\in(\alpha,\beta)\}}{N} = \beta - \alpha.
\]
\end{theorem}

Let $f(x),g(x)\in\mathbb{Z}[x]$ be a pair of primitive, irreducible polynomials. Denote by $g_n(x)$ the reduction of $g(x)$ modulo $n$, by $s(n)$ the number of solutions of the equation $g_n(x) = 0(n)$, and let $(B_j^n)_{j=1}^{s(n)}$ be the sequence of normalized roots of $g_n(x) = 0(n)$.

For each $n$, we define a rectangle $R^n$, which is a set of ordered pairs of normalized roots, as follows.
\[
R^n = (A_i^n)_{i=1}^{r(n)}\times (B_j^n)_{j=1}^{s(n)}.
\]
Taking the dictionary order $(i,j)$, (or any other arbitrary order), we induce an ordering of the elements of $R^n$. As before, we define a series $Z$ to be the concatenation of the rectangular sequences $(R^n_{i,j})_{i=1,j=1}^{r(n)s(n)}$, according to the left dictionary order $(n,i,j)$.
\[
Z = (...,(A_1^n,B_1^n),...,(A_{r(n)}^n,B_{s(n)}^n),(A_1^{n+1},B_1^{n+1}),...).
\]
As before, we denote the $i$'th element of the series $Z$ as $Z_i$. Our main theorem is that the series $Z$ is uniformly distributed in $[0,1]^2$.
\begin{theorem}
\label{thm:main_theorem}
Let $f(x),g(x)\in\mathbb{Z}[x]$ be a pair of primitive, irreducible polynomials of degree greater than one, and let $Z$ be the series of normalized root pairs of $f(x)$ and $g(x)$ (as constructed above). Let $0\le \alpha < \beta\le 1$, and $0\le \gamma < \delta\le 1$ be a quadruple of real numbers. Then
\[
\dfrac{\#\{i\in \mathbb{N}\mid Z_i\in(\alpha,\beta)\times (\gamma,\delta)\}}{N} \longrightarrow (\beta - \alpha)(\delta - \gamma).
\]
\end{theorem}

Joint equidistribution does not persist if one replaces all integer moduli with prime moduli. For example, if one fixes the pair of primitive irreducible polynomials $f(x)$ and $g(x)$ to be $f(x) = g(x) = x^2 + 1$, the polynomial(s) is reducible modulo $p$ if and only if $p=2$ or $p$ is congruent to $1$ modulo $4$. In which case, it has a pair of roots $\pm \nu \mod p$. Therefore, generically, the rectangle of normalized root pairs to the modulus $p$ has the form
\[
R^p = \{(\dfrac{\nu}{p},\dfrac{\nu}{p}),(\dfrac{\nu}{p},\dfrac{p-\nu}{p}),(\dfrac{p-\nu}{p},\dfrac{\nu}{p}),(\dfrac{p-\nu}{p},\dfrac{p-\nu}{p})\},
\]
implying that all normalized roots are concentrated on the diagonals $y=x,y=1-x$, and are therefore not equidistributed in the unit torus. This is a significant difference from the one-dimensional situation, where it is conjectured that for any irreducible polynomial of degree greater than one, that if we restrict to prime moduli, then the roots modulo $p$ are still uniformly distributed, as has been proven for the quadratic case by Duke, Friedlander and Iwaniec \cite{DFI}, see also \cite{TOTH}.

In a forthcoming paper~\cite{KS}, Kowalski and Soundararajan prove a very general result that is closely related to ours when restricted to squarefree moduli.

\subsection*{Acknowledgements}
This result is part of a project that received funding from the European Research Council (ERC) under the European Union's Horizon 2020 research and innovation programme (Grant agreement No. 786758).

\section{Weyl's criterion}
In~\cite{HOOLEY}, Hooley proves Theorem~\ref{thm:hooley_main} by appealing to Weyl's criterion, which states that a sequence $X$ of real numbers is uniformly distributed in the unit interval $[0,1]$ if and only if for all $h\in\mathbb{Z}\setminus \{0\}$ one has
\[
\lim_{M\rightarrow\infty}\dfrac{1}{M}\sum_{n\le M}e(hX_n) = 0,
\]
where $e(x):=e^{2\pi ix}$. Similarly, using a 2-dimensional Weyl criterion, Theorem~\ref{thm:main_theorem} is equivalent to:
\begin{theorem}
\label{thm:main_reformulation}
For all $(h_1,h_2)\in\mathbb{Z}^2\setminus(0,0)$, one has
\[
\lim_{M\rightarrow\infty}\dfrac{1}{M}\sum_{n=1}^Me(h_1Z_n^1 + h_2Z_n^2) = 0.
\]
\end{theorem}

In section 4, we show that Theorem~\ref{thm:main_reformulation} is a consequence of the following statements.
\begin{theorem}
\label{thm:reduction}
\textit{(The Main Theorem)} Let $(h_1,h_2)\in\mathbb{Z}^2\setminus(0,0)$. Then
\[
\lim_{x\rightarrow\infty}\dfrac{1}{\sum_{n\le x}r(n)s(n)}\sum_{n\le x}\sum_{\substack{0\le \mu < n \\ f(\mu)\equiv 0(n)}}e(h_1\dfrac{\mu}{n})\sum_{\substack{0\le \nu < n \\ g(\nu)\equiv 0(n)}}e(h_2\dfrac{\nu}{n}) = 0.
\]
\end{theorem}

\begin{lemma}
\label{lemma:counting_lemma}
\textit{(The Counting Lemma)}
With $r(n)$ and $s(n)$ counting the number of roots of the reductions of $f(x)$ and $g(x)$ modulo $n$, respectively, one has
\[
\dfrac{x\prod_{p\le x}(1+\frac{r(p)s(p)}{p})}{\log x\sum_{n\le x}r(n)s(n)} \ll 1.
\]
\end{lemma}
A proof of the counting Lemma may be found in Section~\ref{proof_of_counting_lemma}. A proof that the Main Theorem and the Counting Lemma imply Theorem~\ref{thm:main_theorem} may be found in Section~\ref{completing_the_argument}.

\section{Proof of the Main Theorem}
Our goal here is to prove Theorem~\ref{thm:reduction}. Denote by $S(a;n)$ the exponential sum
\[
S(a;n) = \sum_{\substack{0\le \mu < n \\ f(\mu)\equiv 0(n)}}e(a\dfrac{\mu}{n}).
\]
Similarly, we denote by $S(h_1,h_2;n)$ the exponential sum
\[
S(h_1,h_2;n) = \sum_{\substack{0\le \mu < n \\ f(\mu)\equiv 0(n)}}e(h_1\dfrac{\mu}{n})\sum_{\substack{0\le \nu < n \\ g(\nu)\equiv 0(n)}}e(h_2\dfrac{\nu}{n}).
\]
In these notations, Theorem~\ref{thm:hooley_main} may be restated as
\[
\sum_{n\le x}S(h;n) = o_h(x),
\]
and our goal in this section is to prove that
\[
\sum_{n\le x}S(h_1,h_2;n) = o_{h_1,h_2}(x).
\]

It is notable that Hooley manages to deduce a bound on the sum $|\sum_{n\le x}S(h;n)|$ by bounding the sum $\sum_{n\le x}|S(h;n)|$, implying that his analysis doesn't take into consideration any potential cancellations between the different exponential sums $S(h;n)$. 

Assuming without loss of generality that $h_1\neq 0$, we will deduce a bound on our sum by constructing an upper bound on the sum
\[
|\sum_{n\le x}S(h_1,h_2;n)|\le \sum_{n\le x}s(n)|S(h_1;n)|.
\]

Since from this point onward our analysis doesn't take $h_2$ into account, we will denote $h_1$ simply as $h$ to simplify the notation.

\subsection{Some preliminary lemmas}
We state some Lemmas from Hooley's paper~\cite{HOOLEY}:
\begin{lemma}
\label{lemma:hooley_1}
\cite[Lemma 1]{HOOLEY}
We have,
\[
\sum_{a=1}^{n}|S(ah;n)|^2 = O(r(n)n\gcd(h,n)).
\]
\end{lemma}

The next Lemma shows that $S(h;n)$ has a twisted multiplicativity property.
\begin{lemma}
\label{lemma:hooley_2}
\cite[Lemma 2]{HOOLEY}
If $\gcd(n,n')=1$, we have
\[
S(h;n)S(h';n') = S(hn' + h'n;nn').
\]
\end{lemma}

\begin{lemma}
\label{lemma:hooley_3}
\cite[Lemma 3]{HOOLEY}
If $\gcd(n,n')=1$, then
\[
S(h;nn') = S(h\overline{n}';n)S(h\overline{n};n'),
\]
where $\overline{n}, \overline{n}'$ satisfy the congruences
\[
n\overline{n} \equiv 1\mod n',\quad n'\overline{n}'\equiv 1\mod n.
\]
\end{lemma}

We consider some general properties of $r(n),s(n)$. We write the properties for $r(n)$, however they hold similarly for $s(n)$.
\begin{lemma}
\label{lemma:hooley_4}
\cite[Lemma 4]{HOOLEY}
Let $n,k$ denote natural numbers, and $p$ a prime number. One has
\begin{itemize}
    \item $r(n)$ is a multiplicative function of $n$.
    \item if $p\nmid disc(f)$, then $r(p) = r(p^k)\le \deg(f)$.
    \item $r(p^k) = O(1)$.
    \item $r(n) = O(\deg(f)^{\omega(n)})$, where $\omega(n)$ is the distinct prime factor counting function.
\end{itemize}
\end{lemma}

\begin{lemma}
\label{lemma:hooley_6}
\cite[Lemma 6]{HOOLEY}
Let $L/\mathbb{Q}$ be a Galois extension of the rationals. Let the letter $q$ denote those primes in $\mathbb{Z}$, such that $q$ splits completely in $L$, and denote the degree of the extension of $L/\mathbb{Q}$ by $[L:\mathbb{Q}]$. Then we have
\[
\prod_{q\le x}(1 + \dfrac{1}{q}) = \Omega(\log^{\frac{1}{[L:\mathbb{Q}]!}}(x)).
\]
\end{lemma}

\subsection{An adaptation of Hooley's argument}
Before we commence the analysis of the sum
\[
\sum_{n\le N}s(n)|S(h;n)|,
\]
similarly to Hooley, we note that the analysis is dependent in part on the use of numbers whose prime factors are restricted by certain conditions. To avoid interrupting the argument later, we define a notion for these numbers, and state a lemma of Hooley on their properties.

Let
\[
X=x^{B/\log\log x},
\]
where
\[
B = \dfrac{1}{24e\deg(f)\deg(g)}.
\]
Then let the subscript 1 attached to appropriate letters indicate that these letters denote either the number 1, or numbers composed entirely of prime factors not exceeding $X$; let also the subscript $2$ indicate, similarly, either 2 or numbers composed entirely of prime factors exceeding $X$. Furthermore, for any integer denoted by the letter $n$, we understand $n_1$ and $n_2$ to be defined by the equation $n=n_1n_2$.

\begin{lemma}
\label{lemma:hooley_8}
\cite[Lemma 8]{HOOLEY}
If $a,\lambda$ and $y$ satisfy the conditions $y\ge x^{2/3}$, $\lambda\le x^{1/3}$ and $(a,\lambda)=1$, then
\[
\sum_{\substack{l_2\le y\\ l_2\equiv a(\lambda)}}1 = O\big(\dfrac{y}{\phi(\lambda)\log x}\big)
\]
\end{lemma}

The estimation of $\sum_{n\le N}s(n)|S(b,n)|$ may now be commenced.
\[
\sum_{n\le x}s(n)|S(h;n)| = \sum_{n_1n_2\le x}s(n_1)|S(h\overline{n}_2;n_1)|
s(n_2)|S(h\overline{n}_1;n_2)|
\]
\[
=
\sum_{n_1\le x^{\frac{1}{3}}} + \sum_{n_1 > x^{\frac{1}{3}}} = \sum_1 + \sum_2.
\]
The sum $\sum_2$ is examined first. By Lemma~\ref{lemma:hooley_4} we have
\[
\sum_2 = O\big(\sum_{\substack{n\le x\\n_1 > x^{\frac{1}{3}}}}r(n)s(n)\big) = O\big(\sum_{\substack{n\le x\\n_1 > x^{\frac{1}{3}}}}
(\deg(f)\deg(g))^{\omega(n)}\big).
\]
In~\cite[page 6]{HOOLEY}, Hooley shows:
\begin{claim}
\cite[Page 6]{HOOLEY} If $0 < c < \dfrac{1}{12eB} = 2\deg(f)\deg(g)$, then
\[
O\big(\sum_{\substack{n\le x\\n_1 > x^{\frac{1}{3}}}}
c^{\omega(n)}\big) = O(\dfrac{x}{\log x}).
\]
\end{claim}

This is true in particular for $c=\deg(f)\deg(g)$, $\sum_2 = O(\dfrac{x}{\log x})$.

We turn our attention to $\sum_1$. We have
\[
\sum_1 = \sum_{\substack{n_1n_2 \le x\\n_1\le x^{1/3}}}s(n)|S(h\overline{n}_2;n_1)||S(h\overline{n}_1;n_2)|
\]
\[
=
O\big(\sum_{\substack{n_1n_2\le x\\n_1\le x^{1/3}}}
s(n_1)s(n_2)r(n_2)|S(h\overline{n}_2;n_1)|\big)
\]
\[
=
O\big(\sum_{n_1\le x^{1/3}}s(n_1)\sum_{n_2\le x/n_1}s(n_2)r(n_2)|S(h\overline{n}_2;n_1)|\big)
\]
\[
=
O\big(\sum_{n_1\le x^{1/3}}s(n_1)\theta(\dfrac{x}{n_1}, n_1)\big),
\]
where for $y\ge x^{2/3}$ and $n_1\le x^{1/3}$ we define $\theta(y,n_1)$ slightly different than Hooley, by
\[
\theta(y,n_1) = \sum_{n_2\le y}s(n_2)r(n_2)|S(h\overline{n}_2;n_1)|.
\]
By the Cauchy-Schwarz inequality,
\[
\theta^2(y,n_1)\le (\sum_{n_2\le y}s^2(n_2)r^2(n_2))(\sum_{n_2\le y}|S(h\overline{n}_2;n_1)|^2) = (\sum_5)(\sum_6).
\]
\begin{claim}
\cite[Page 8]{HOOLEY} The sum $\sum_{n_2\le y}D^{\omega(n_2)}$ is at most $O(\dfrac{y(\log\log x)^D}{\log x})$.
\end{claim}

Applying Hooley's claim for $D = \deg(f)\deg(g)$, we have
\[
\sum_5 = O\big(\dfrac{y(\log\log x)^D}{\log x}\big).
\]

Next, we analyze $\sum_6$.
\[
\sum_6 = \sum_{\substack{0 < a\le n_1\\\gcd(a,n_1)=1}}|S(ha;n_1)|^2\sum_{\substack{n_2\le y\\n_2\equiv a(n_1)}}1.
\]
The inner sum is an implication of Lemma~\ref{lemma:hooley_8}. Plugging this estimate into $\sum_6$ we get
\[
\sum_6 = O\big(\dfrac{y\log\log x}{\phi(n_1)\log x}\sum_{\substack{0 < a\le n_1\\\gcd(a,n_1)=1}}|S(ha,n_1)|^2\big).
\]
Applying Lemma~\ref{lemma:hooley_3}, we obtain the estimate
\[
\sum_6 = O\big(\dfrac{y\log\log x}{\phi(n_1)\log x}r(n_1)n_1\gcd(h,n_1)\big).
\]
Allowing the constants in our notation to depend on $h$, 
our estimate of $\theta(y,n_1)$ becomes
\[
\theta(y,n_1) = O\big(\dfrac{y(\log\log x)^{(D+1)/2}}{\log x}\cdot\dfrac{r(n_1)^{1/2}n_1^{1/2}}{\phi(n_1)^{1/2}}\big).
\]
Plugging back to $\sum_1$, we have
\[
\sum_1\le O\big(\dfrac{x(\log\log x)^{(D+1)/2}}{\log x}\cdot\sum_{n_1\le x^{1/3}}\dfrac{s(n_1)r(n_1)^{1/2}}{n_1^{1/2}\phi(n_1)^{1/2}}\big)
\]
\[
\implies
\sum_1\le O\big(\dfrac{x(\log\log x)^{(D+1)/2})}{\log x}\cdot\sum_7\big).
\]

We proceed to analyze $\sum_7$.
\[
\sum_7 = \sum_{n_1\le x^{1/3}}\dfrac{s(n_1)r(n_1)^{1/2}}{n_1^{1/2}\phi(n_1)^{1/2}}\le
\sum_{n\le x}\dfrac{s(n)r(n)^{1/2}}{n^{1/2}\phi(n)^{1/2}}
\]
\[
\le
\prod_{p\le x}\big(1 + (1-\dfrac{1}{p})^{-1/2}\sum_{k=1}^{\infty}\dfrac{s(p^k)r(p^k)^{1/2}}{p^k}\big)
\]
\[
\le
\prod_{p\le x}\big(1 + (1-\dfrac{1}{p})^{-1}\sum_{k=1}^{\infty}\dfrac{s(p^k)r(p^k)^{1/2}}{p^k}\big)
\]
\[
\ll
\prod_{\substack{p\le x\\\gcd(p,disc(f)disc(g)) = 1}}\big(1 + (1-\dfrac{1}{p})^{-1}\sum_{k=1}^{\infty}\dfrac{s(p)r(p)^{1/2}}{p^k}\big)
\]
\[
\le
\prod_{\substack{p\le x \\ \gcd(p,disc(f)disc(g)) = 1}}\big(1 + (1-\dfrac{1}{p})^{-2}\dfrac{s(p)r(p)^{1/2}}{p}\big)
\]
\[
\le
O\bigg(
\prod_{\substack{p\le x \\ \gcd(p,disc(f)disc(g)) = 1}}\big(1 + \dfrac{s(p)r(p)^{1/2}}{p}\big)
\bigg).
\]
Let us denote by $\mathcal{Q}$ the set of primes satisfying $r(q) = \deg(f), s(q) = \deg(g)$. Then
\[
\sum_7
\le
O\bigg(
\prod_{\substack{p\le x \\ \gcd(p,disc(f)disc(g)) = 1}}\big(1 + \dfrac{s(p)r(p)}{p}\big)
\prod_{\substack{q\le x \\ \gcd(q,disc(f)disc(g)) = 1 \\ q\in \mathcal{Q}}}\big(\dfrac{1 + \frac{s(q)r(q)^{1/2}}{q}}{1 + \frac{s(q)r(q)}{q}}\big)
\bigg),
\]

\[
\sum_7 \le
O
\bigg(
\prod_{\substack{p\le x \\ \gcd(p,disc(f)disc(g)) = 1}}
(1 + \dfrac{s(p)r(p)}{p})
\prod_{\substack{q\le x \\ \gcd(q,disc(f)disc(g)) = 1 \\ q\in \mathcal{Q}}}
(1 + \dfrac{1}{q})^{\deg(g)(\deg(f)^{1/2} - \deg(f))}
\bigg).
\]

Ideally, we would like to apply Lemma~\ref{lemma:hooley_6} to obtain a lower bound on 
\[
\prod_{\substack{q\le x \\ \gcd(q,disc(f)disc(g)) = 1 \\ q\in\mathcal{Q}}}(1 + \dfrac{1}{q}).
\]
However, since Lemma~\ref{lemma:hooley_6} requires the primes $q$ to be the totally split primes in some Galois extension of the rationals, we will be required to prove first that this is indeed the case.

Before we proceed, let us fix the following notation. For an arbitrary number field $F$, denote by $Spl(F)$ the set of primes $p\in Spec(\mathbb{Z})$ which are totally split in $F$.

\begin{claim}
\label{claim:algebraic}
Let $\mathcal{Q}$ be the set of primes satisfying $r(q) = \deg(f), s(q) = \deg(g)$ (as above). Then, there exists some Galois extension $L$ of $\mathbb{Q}$ such that the symmetric difference of $L$ and $\mathcal{Q}$ is finite.

In other words, for all but finitely many primes $q$, one has
\[
q\in \mathcal{Q} \iff q\in Spl(L).
\]
\end{claim}

\begin{proof}
Let $K/\mathbb{Q}$ and $T/\mathbb{Q}$ be the two extensions of the rationals obtained by appending to $\mathbb{Q}$ a root of $f(x)$ and a root of $g(x)$, respectively. Let $L$ denote the Galois closure of their compositum, $KT$, then we will show that $L$ satisfies the requirements of the claim.

The following are well known theorems in algebraic number theory.

\begin{theorem}
Let $S$ and $R$ denote two number fields, then $Spl(SR) = Spl(S)\cap Spl(R)$.
\end{theorem}

\begin{proof}
This is a variant of Neukirch's~\cite{NEUKIRCH} exercise 3 in Section 8 of Chapter 1.
\end{proof}
An immediate consequence of this theorem is that 
\[
Spl(KT) = Spl(K)\cap Spl(T).
\] 

Another theorem we use is the following.
\begin{theorem}
Let $S$ be a number field, and let $R$ denote its Galois closure, then $Spl(S) = Spl(R)$.
\end{theorem}

\begin{proof}
This is a variant of Neukirch's~\cite{NEUKIRCH} exercise 4 in Section 8 of Chapter 1.
\end{proof}
An immediate consequence of this theorem is that $Spl(L) = Spl(KT)$, and combining with the previous conclusion, we get
\[
Spl(L) = Spl(K)\cap Spl(T).
\] 

Another theorem we borrow is a classical theorem of Dedekind.
\begin{theorem}
\cite[Section 4, c.f.~\cite{KCONRAD}]{DEDEKIND} Let $S$ be a number field and $\alpha\in \mathcal{O}_S$, such that $S = \mathbb{Q}(\alpha)$. Let
$h(x)$ be the minimal polynomial of $\alpha$ in $\mathbb{Z}[x]$. For any prime $p$ not dividing the index $[\mathcal{O}_S : \mathbb{Z}[\alpha]]$, write
\[
f(x) = \pi_1(x)^{e_1}\cdot ... \cdot \pi_g(x)^{e_g}\mod p,
\]
where the $\pi_i(x)$'s are distinct monic irreducibles in $\mathbb{F}_p[x]$. Then $(p) = p\mathcal{O}_S$ factors into prime ideals as
\[
(p) = \mathfrak{p}_1^{e_1}\cdot ... \cdot \mathfrak{p}_g^{e_g},
\]
where there is a bijection between the $\mathfrak{p}_i$’s and $\pi_i(x)$’s such that
\[
N(\mathfrak{p}_i) = p^{\deg(\pi_i)}.
\]
In particular, this applies for all $p$ if $\mathcal{O}_S = \mathbb{Z}[\alpha]$, i.e. if $\mathcal{O}_S$ is monogenic.
\end{theorem}
\begin{proof}
See Conrad's document on Dedekind's theorem~\cite{KCONRAD}, or Lang's book~\cite[I, Proposition 25]{LANG}.
\end{proof}
An immediate conclusion is that for all but finitely many primes $q$,
\[
r(q) = \deg(f) \iff q\in Spl(K),
\]
and similarly
\[
s(q) = \deg(g) \iff q\in Spl(T).
\]
Taking an intersection of the two sets and combining with the previous statements proves Claim~\ref{claim:algebraic}.
\end{proof}

Returning to the analysis of $\sum_7$, the previous claim implies that for large enough $x$, 
\[
\prod_{\substack{q\le x \\ \gcd(q,disc(f)disc(g)) = 1 \\ q\in\mathcal{Q}}}
(1 + \dfrac{1}{q}) \quad\text{and} \prod_{\substack{q\le x \\ \gcd(q,disc(f)disc(g)) = 1 \\ q\in Spl(L)}}
(1 + \dfrac{1}{q})
\]
differ at most by a constant factor.

Therefore, since $\prod_{\substack{q\le x \\ \gcd(q,disc(f)disc(g)) = 1 \\ q\in Spl(L)}}
(1 + \dfrac{1}{q})$ and $\prod_{\substack{q\le x \\ q\in Spl(L)}}
(1 + \dfrac{1}{q})$ also differ at most by a constant factor, we have
\[
\prod_{\substack{q\le x \\ \gcd(q,disc(f)disc(g)) = 1 \\ q\in\mathcal{Q}}}
(1 + \dfrac{1}{q}) \gg
\prod_{\substack{q\le x \\ q\in Spl(L)}}
(1 + \dfrac{1}{q}).
\]

Therefore,
\[
\sum_7
\le
O
\bigg(
\prod_{p\le x}
(1 + \dfrac{s(p)r(p)}{p})
\prod_{\substack{q\le x \\ q\in Spl(L)}}
(1 + \dfrac{1}{q})^{\deg(g)(\deg(f)^{1/2} - \deg(f))}
\bigg).
\]
Applying Lemma~\ref{lemma:hooley_6} for $L$, we find that
\[
\prod_{q\le x}
(1 + \dfrac{1}{q})^{\deg(g)(\deg(f)^{1/2} - \deg(f))} = O(\log^{-\delta}(x)),
\]
where $\delta > 0$ denotes the ratio
\[
\delta = \dfrac{\deg(g)(\deg(f) - \deg(f)^{1/2})}{[L:\mathbb{Q}]!}.
\]
Then we obtain
\[
\sum_7\le O(\prod_{p\le x}
(1 + \dfrac{s(p)r(p)}{p})\log^{-\delta}(x)),
\]
and plugging this into our previous bound on $\sum_1$, we obtain
\[
\sum_1 \le O\big(\prod_{p\le x}
(1 + \dfrac{s(p)r(p)}{p})\dfrac{x(\log\log x)^{(D+1)/2}}{\log^{1+\delta}(x)}\big).
\]

Combining the estimates of $\sum_1$ and $\sum_2$, we see that
\[
|\sum_{n\le x}S(h_1,h_2;n)| = O_{h_1,h_2}\bigg(\dfrac{x(\log\log x)^{(D+1)/2}\prod_{p\le x}
(1 + \dfrac{s(p)r(p)}{p})}{\log^{1 + \delta}(x)}\bigg).
\]

Therefore,
\[
|\dfrac{1}{\sum_{n\le x}r(n)s(n)}\sum_{n\le x}S(h_1,h_2;n)| \ll_{h_1,h_2}
\dfrac{x\prod_{p\le x}(1 + \dfrac{s(p)r(p)}{p})}
{\log x\sum_{n\le x}r(n)s(n)}\cdot\dfrac{(\log\log x)^{(D+1)/2}}{\log^{\delta} x}.
\]

By Lemma~\ref{lemma:counting_lemma}, one has
\[
\dfrac{x\prod_{p\le x}(1 + \dfrac{s(p)r(p)}{p})}
{\log x\sum_{n\le x}r(n)s(n)} = O(1),
\]
and since trivially
\[
\lim_{x\rightarrow\infty}\dfrac{(\log\log x)^{(D+1)/2}}{\log^{\delta} x} = 0,
\]
we find that
\[
\lim_{x\rightarrow\infty}\dfrac{1}{\sum_{n\le x}r(n)s(n)}\sum_{n\le x}S(h_1,h_2;n) = 0,
\]
which proves the Main Theorem.

\section{Proving the Counting Lemma}
\label{proof_of_counting_lemma}
The Counting Lemma states that
\[
\dfrac{x\prod_{p\le x}(1 + \frac{s(p)r(p)}{p})}{\log x\sum_{n\le x}r(n)s(n)} \ll 1.
\]

In Appendix A of~\cite{OPERA}, Friedlander and Iwaniec prove the following lemma, useful for our purposes
\begin{lemma}
(Theorem A.4) Let $f(n)$ be some multiplicative function supported on squarefree integers and suppose there are some $a\ge 0$ and $b\ge 1$ such that $f$ satisfies
\[
\sum_{y< p\le x}\dfrac{f(p)\log p}{p} \le a\log\dfrac{x}{y} + b,
\]
for all $2\le y\le x$. Moreover, suppose that
\[
\sum_{p\le y}f(p)\log p \gg y
\]
for all sufficiently large $y$. Then for all $x\ge 2$,
\[
\sum_{n\le x}f(n) \gg \dfrac{x}{\log x}\prod_{p\le x}(1 + \frac{f(p)}{p}).
\]
\end{lemma}

First, we observe that this theorem implies the Counting Lemma. For if we interchange $f(n)$ with $r(n)s(n)\mu^2(n)$, where $\mu(n)$ is the M\"{o}bius function, then for all $x\ge 2$:
\[
\sum_{n\le x}r(n)s(n) \ge \sum_{n\le x}r(n)s(n)\mu^2(n) \gg \dfrac{x}{\log x}\prod_{p\le x}(1 + \frac{r(p)s(p)}{p}) \implies \dfrac{x\prod_{p\le x}(1 + \frac{r(p)s(p)}{p})}{\log x \sum_{n\le x}r(n)s(n)}\ll 1.
\]
Therefore, it suffices to prove that $r(n)s(n)$ satisfies the following two properties (as in the theorem):
\begin{itemize}
    \item there exist some $a\ge 0$ and $b\ge 1$ such that for all $2\le y\le x$ one has $\sum_{y< p\le x}\dfrac{r(p)s(p)\log p}{p} \le a\log\dfrac{x}{y} + b$,
    \item for all sufficiently large $y$, one has $\sum_{p\le y}r(p)s(p)\log p \gg y$.
\end{itemize}
We begin by proving the first property.
\begin{claim}
There exist some $a\ge 0$ and $b\ge 1$ such that for all $2\le y\le x$ one has
\[
\sum_{y< p\le x}\dfrac{r(p)s(p)\log p}{p} \le a\log\dfrac{x}{y} + b.
\]
\end{claim}
\begin{proof}
This claim is an immediate implication of Mertens' first theorem. Clearly, since $r(p),s(p)$ are both upper bounded by a constant, one has
\[
\sum_{p\le x}\dfrac{r(p)s(p)\log p}{p} \ll \sum_{p\le x}\dfrac{\log p}{p},
\]
and by Mertens' first theorem,
\[
\sum_{p\le x}\dfrac{\log p}{p} = \log x + O(1),
\]
implying that there is some constant $a > 0$ such that
\[
\sum_{p\le x}\dfrac{r(p)s(p)\log p}{p} \le a\log x + O(1).
\]
Subtracting the estimation for $\sum_{p\le y}\dfrac{\log p}{p}$ from the estimation for $\sum_{p\le x}\dfrac{\log p}{p}$ one obtains
\[
\sum_{y < p\le x}\dfrac{r(p)s(p)\log p}{p} \ll a\log \dfrac{x}{y} + O(1),
\]
which proves the claim.
\end{proof}
\begin{claim}
For all sufficiently large $y$, one has $\sum_{p\le y}r(p)s(p)\log p \gg y$.
\end{claim}
\begin{proof}
By Claim~\ref{claim:algebraic}, there exists number field $L$ which is Galois over $\mathbb{Q}$ such that for all but finitely many primes $p$,
\[
r(p) = \deg(f), s(p) = \deg(q) \iff p \in Spl(L).
\]
Therefore,
\[
\sum_{p\le y}r(p)s(p)\log p \ge \sum_{\substack{p\le y\\ p\in Spl(L)}}r(p)s(p)\log p \ge \sum_{\substack{p\le y\\ p\in Spl(L)}}\log p + O(1) \gg y.
\]
Where the first inequality is because we restrict our sum to a subset, the second inequality is because over this subset $r(p) = \deg(f)$, $s(p) = \deg(g)$ except for at most finitely many primes, accounted by the $O(1)$ term, and the third inequality is an immediate consequence of the Chebotarev density theorem.
\end{proof}

\section{Completing the argument}
\label{completing_the_argument}
To complete the proof of Theorem~\ref{thm:main_theorem} we need to show, according to Theorem~\ref{thm:main_reformulation}, that for all $(h_1,h_2)\in\mathbb{Z}^2\setminus (0,0)$, the sequence of normalized solution pairs $Z$ satisfies:
\[
\dfrac{1}{M}\sum_{i=1}^Me(h_1Z_i^1 + h_2Z_i^2) = o(1).
\]
Let us denote by $N = N(M)$ the largest modulus for which some normalized root pair (belonging to this modulus) appears in this sum. Then trivially we have
\[
M\le \sum_{n=1}^Nr(n)s(n) \le M + r(N)s(N).
\]
From the triangle inequality, we have
\[
|\sum_{i=1}^Me(h_1Z_i^1 + h_2Z_i^2)| \le |\sum_{n=1}^NS(h_1,h_2;n)| + r(N)s(N),
\]
implying that
\[
\dfrac{1}{M}|\sum_{i=1}^Me(h_1Z_i^1 + h_2Z_i^2)| \le \dfrac{1}{\sum_{n\le N}r(n)s(n) - r(N)s(N)}|\sum_{n=1}^NS(h_1,h_2;n)| + \dfrac{r(N)s(N)}{\sum_{n\le N}r(n)s(n) - r(N)s(N)} = \sum_1 + \sum_2.
\]

Analyzing $\sum_1$, we have
\[
\dfrac{1}{\sum_{n\le N}r(n)s(n) - r(N)s(N)}|\sum_{n=1}^NS(h_1,h_2;n)| = \dfrac{\sum_{n\le N}r(n)s(n)}{\sum_{n\le N}r(n)s(n) - r(N)s(N)}\cdot\dfrac{1}{\sum_{n\le N}r(n)s(n)}|\sum_{n=1}^NS(h_1,h_2;n)|,
\]
and since the Counting Lemma implies that $\sum_{n\le N}r(n)s(n) \gg \dfrac{N}{\log N}$, and Lemma~\ref{lemma:hooley_4} implies that $r(N)s(N) = o(\dfrac{N}{\log N})$, one has $r(N)s(N) = o(\sum_{n\le N}r(n)s(n))$, implying that
\[
\dfrac{\sum_{n\le N}r(n)s(n)}{\sum_{n\le N}r(n)s(n) - r(N)s(N)} \longrightarrow 1,
\]
and in particular, $|\dfrac{\sum_{n\le N}r(n)s(n)}{\sum_{n\le N}r(n)s(n) - r(N)s(N)}| \ll 1$, and since the Main Theorem, i.e. Theorem~\ref{thm:main_reformulation}, says that
\[
\dfrac{1}{\sum_{n\le N}r(n)s(n)}|\sum_{n=1}^NS(h_1,h_2;n)| = o(1),
\]
one finds that $\sum_1 = o(1)$. That $\sum_2 = o(1)$ is an immediate implication of the fact that $r(N)s(N) = o(\sum_{n\le N}r(n)s(n))$. This concludes the proof of Theorem~\ref{thm:main_theorem}.

\end{document}